\documentclass[11pt]{article}
\usepackage{amssymb}
\usepackage{amsmath}
\usepackage{amsthm}
\usepackage{amsopn}
\usepackage{comment}
\usepackage{graphicx}
\usepackage{a4wide}

\numberwithin{equation}{section}

\swapnumbers

\newtheorem{satz}{Satz}[section]

\newtheorem{theorem}[satz]{Theorem}

\newtheorem{lemma}[satz]{Lemma}

\theoremstyle{definition}

\newtheorem{remark}[satz]{Remark}

\DeclareMathOperator{\E}{{\mathbf E}}

\DeclareMathOperator{\R}{{\mathbb R}}

\DeclareMathOperator{\N}{{\mathbb N}}

\DeclareMathOperator{\PP}{{\mathbb P}}

\DeclareMathOperator{\diag}{diag}

\DeclareMathOperator{\Var}{Var}

\providecommand{\eps}{\varepsilon}
\renewcommand{\phi}{\varphi}
\renewcommand{\theta}{\vartheta}
\renewcommand{\subset}{\subseteq}

\renewcommand{\cdot}{{\scriptstyle \bullet} }
\providecommand{\abs}[1]{\lvert #1 \rvert}
\providecommand{\norm}[1]{\lVert #1 \rVert}

\renewcommand{\le}{\leqslant}
\renewcommand{\ge}{\geqslant}

\newcommand{\e}{\varepsilon }
\title{Model selection in sparse heterogeneous framework}
\author{Laurent Cavalier \footnote{Universit\'e Aix-Marseille, LATP, CMI, 39 rue Joliot-Curie, F-13453 Marseille cedex 13, France, cavalier@cmi.univ-mrs.fr} \and Markus Rei\ss
\footnote{Institut f\"ur Mathematik, Humboldt-Universit\"at zu Berlin, Unter den Linden 6, D-10099 Berlin, Germany, mreiss@math.hu-berlin.de}}
\date{\today}

\begin{document}

\maketitle

\begin{abstract}
We consider a Gaussian sequence space model $X_{\lambda}=f_{\lambda} + \xi_{\lambda},$
where $\xi $ has a diagonal covariance matrix $\Sigma=\diag(\sigma_\lambda ^2)$.
We consider the situation where the parameter vector $(f_{\lambda})$ is sparse.
Our goal is to estimate the unknown parameter by a model selection approach.
The heterogenous case is much more involved than the direct model.
Indeed, there is no more symmetry inside the stochastic process that one needs to control since each empirical coefficient has its own variance.
The problem and the penalty do not only depend on the number of coefficients that one selects, but also on their position.
This appears also in the minimax bounds where the worst coefficients will go to the larger variances.
However, with a careful and explicit choice of the penalty we are able to select the correct coefficients
and get a sharp non-asymptotic control of the risk of our procedure.
Some simulation results are provided.
\end{abstract}

\section{Introduction}

\subsection{Motivation and main results}

We consider the following sequence space model
\begin{equation}\label{EqModel}
X_{\lambda}=f_{\lambda} + \xi_{\lambda},\quad \lambda\in\Lambda
\end{equation}
where $(f_\lambda)$ are the coefficients of a signal and the noise $(\xi_{\lambda})\sim {\cal N}(0,\Sigma)$ has a diagonal covariance matrix $\Sigma=\diag(\sigma_\lambda^2)$.
This heterogeneous model may appear in several frameworks where the variance is fluctuating,
for example in heterogeneous regression, coloured noise, fractional Brownian motion models or statistical inverse problems, for which the general literature
is quite exhaustive \cite{JS,AS,CGPT,C04,CR,CHR,C11,D,HR,JP,R}.
The goal is to estimate the unknown parameter $f_{\lambda}$ by using the observations $(X_{\lambda})$.

Model selection is a core problem in statistics.
One of the main reference in the field dates back to the AIC criterion \cite{A}, but there has been a huge amount of papers
on this subject (e.g., \cite{BM,G02,ABDJ,M,G11,R,WZ}). Model selection is usually linked to the choice of a penalty and its precise choice is the main difficulty in model selection both from a theoretical and a practical perspective.

There is a close relationship between model selection and thresholding procedures, which is addressed e.g. in \cite{ABDJ,G02,M}. The idea is that the search for a ``good penalty'' in model selection is indeed very much related
 to the choice of a ``good threshold'' in wavelet procedures. There exists also a fascinating connection between the false discovery rate control (FDR) and both thresholding and model selection,
as studied in \cite{ABDJ,BH}, which will become apparent later in our paper.

Our main modeling assumption  is that the parameter $(f_{\lambda})$ of interest is sparse.
Sparsity is one of the leading paradigms nowadays and signals with a sparse representation in some basis (for example wavelets)
or functions with sparse coefficients appear in many scientific fields (see \cite{ABDJ,G02,G11,WZ} among many others).

In this paper, we consider the sequence space model with heterogeneous errors.
Our goal is then to select among a family of models the best possible one, by use of a data-driven selection rule.
In particular, one has to deal with the special heterogeneous nature of the observations, and the choice of the penalty must reflect this.
The heterogenous case is much more involved than the direct (homogeneous) model. Indeed, there is no more symmetry inside the stochastic process
that one needs to control, since each empirical coefficient has its own variance.
The problem and the penalty do not only depend on the number of coefficients that one selects, but also on their position.
This also appears in the minimax bounds where the coefficients in the least favourable model will go to the larger variances.
By a careful and explicit choice of the penalty, however, we are able to select the correct coefficients
and get a sharp non-asymptotic control of the risk of our procedure.
Results are also obtained for full model selection and a FDR-type control on a family of thresholds.
In the case of known sparsity $\gamma_n$, we consider a non-adaptive threshold estimator and obtain a minimax upper bound.
This estimator exactly attains the lower bound and is then minimax.
Using our model selection approach, the procedure is almost minimax (up to a factor 2). Moreover, the procedure is fully adaptive.
Indeed, the  sparsity $\gamma_n$ is unknown and we obtain an explicit penalty, valid in the mathematical proofs and directly applicable in simulations.

The paper is organized as follows. In the following Subsection \ref{sec:exa}, we give examples of problems where our heterogeneous model appears.
Section \ref{sec:sel} contains the data-driven procedure and a general result.
In Section \ref{sec:spa}, we consider the sparsity assumptions and obtain theorems for the full subset selection
and thresholding procedures. Section \ref{sec:low} and \ref{sec:upp} are concerned with minimax lower and upper bounds.
In Section \ref{sec:num}, we present numerical results for the finite-sample properties of the methods.

\subsection{Examples}
\label{sec:exa}

\subsubsection*{Heterogeneous regression}

Consider first a model of heterogeneous regression
$$
Y_i=f(x_i)+ \sigma (x_i)\eps_i, \ \ \ i=1,\dots ,n,
$$
where $\eps_i$ are i.i.d. standard Gaussian, but their variance are fluctuating depending
on the design points $x_i$ and $f$ is some spiky unknown function.  In this model $\Lambda =\{1,\dots ,n\}$.
By spiky function we mean that $f(x_i)$ is zero apart from a small subset of all design points $x_i$.
These signals are frequently encountered in applications (though rarely modeled in theoretical statistics), e.g. when measuring absorption spectra in physical chemistry
(i.e. rare well-localised and strong signals) or jumps in log returns of asset prices (i.e. log-price increments which fluctuate at low levels except when larger shocks occur).

\subsubsection*{Coloured noise}

\label{sec:FBM}

Often in applications coloured noise models are adequate. Let us consider here
the problem of estimating an unknown function observed
with a noise defined by some fractional Brownian motion,
\begin{equation}
\label{mod}
dY(t)=f(t)dt + \e dW_{-\alpha }(t),\ \ \ t\in [0,1],
\end{equation}
where $f$ is an unknown $1-$periodic function in $L^2(0,1)$, $\int_0^1f(t)dt$=0, $\e $ is the noise level and
$W_{-\alpha }$ is a fractional Brownian motion, defined by (see  \cite{S}),
\begin{equation}
\label{fracW}
W_{-\alpha }(x)=\int_{-\infty}^x {\frac{(x-t)^{-\alpha }}
{\Gamma (1-\alpha)}dW(t)},
\end{equation}
where $W$ is a Brownian motion, $0\le \alpha <1/2$, $\Gamma (\cdot )$ is the Gamma function. The fractional Brownian motion also appears in  econometric applications to model the long-memory phenomena, e.g. in \cite{CR96}.
The model (\ref{mod}) is close to the standard Gaussian white noise model, which corresponds
to the case $\alpha =0$. Here, the behaviour of the noise is different.

We are not interested in the fractional Brownian motion itself, but
we want to estimate the unknown function $f$ based on the noisy data $Y(t)$, as in \cite{C04,J,W}.

\bigskip
A very important point is linked with the definition of the fractional integration operator.
In this framework, if the function $f$ is supposed to be $1-$periodic, then the natural way is to consider
the periodic version of fractional integration (given in (\ref{frac})), such that
\begin{equation}
\label{frac}
d^{-\alpha } f(x)=\int_{-\infty }^x {\frac{(x-t)^{\alpha -1}} {\Gamma (\alpha )}f(t)dt},
\end{equation}
and thus (see p.135 in \cite{Z}),
\begin{equation}
\label{eigen}
d^{-\alpha } e^{2\pi ikx}=\frac{e^{2\pi ikx}}{(2\pi i k)^{\alpha }}.
\end{equation}

By integration and projection on the cosine (or sine) basis and using (\ref{eigen}),
one obtains the sequence space model (as in \cite{C04}),
$$
X_{\lambda}=f_{\lambda} + \xi_{\lambda},\ \lambda \in \Lambda=\N,
$$
where $\{\xi_{\lambda}\}$ are independent with $(\xi_{\lambda})_\lambda\sim {\cal N}(0,\Sigma)$, where $\Sigma= \diag(\sigma_\lambda^2)$ and
$\sigma_\lambda^{2}=\eps^2(2\pi \lambda)^{2\alpha }$.

\subsubsection*{Inverse problems}

Consider the following framework of a general inverse problem
$$
Y=Af + \eps \ \dot W,
$$
where $A$ is a known injective compact linear bounded operator, $f$ an unknown $d$-dimensional function,
$\dot W$ is a Gaussian white noise and $\e >0$ the noise level.
We will use here the framework of Singular Values Decomposition (SVD), see e.g. \cite{C11}.
Denote by $\phi_\lambda $ the eigenfunctions of the operator $A^*A$ associated with the strictly positive eigenvalues $b_\lambda^2>0$.
Remark that any function $f$ may be decomposed in this orthonormal basis as
$f=\sum_{\lambda \in \Lambda}f_\lambda \phi_\lambda $, where $\lambda \in \Lambda$.

Let $\{\psi_\lambda\}_{\lambda \in \Lambda}$ be the normalized image basis
$
\psi_\lambda = b_\lambda^{-1}A\phi_\lambda .
$
%where $\psi_\lambda \in \mathrm{range}\ (A^*)$.
By projection and division by the singular values, we may obtain the empirical coefficients
$$
b_\lambda^{-1}\langle Y,\psi_\lambda \rangle= b_\lambda^{-1}\langle Af, b_\lambda^{-1}A\phi_\lambda\rangle +
b_\lambda^{-1}\langle \eps \dot W,\psi_\lambda \rangle
=  \langle f,\psi_\lambda \rangle +b_\lambda^{-1}\langle \eps \dot W,\psi_\lambda \rangle .
$$
We then obtain a model in the sequence space (see \cite{CGPT})
$$
X_{\lambda}=f_{\lambda} + \xi_{\lambda},\ \lambda \in \Lambda,
$$
with 
$(\xi_{\lambda})_\lambda\sim {\cal N}(0,\Sigma)$ and $\Sigma=\diag(\e^2 b_\lambda^{-2})$.

%If $A$ is a homogeneous operator, ($d=1$?), $a$ the Degree of Ill-Posedness (DIP),
%we have
%$$
%\Sigma_{(j,k),(j',k')}\sim 2^{j-a}\eps^2!!
%$$
%Remark that the matrix $\Sigma $ is also sparse.

\section{Data-driven-subset selection}
\label{sec:sel}

We consider the sequence space model \eqref{EqModel} for coefficients of an unknown $L^2$-function $f$ with respect to an orthornormal system $(\psi_\lambda)$. The estimator over an arbitrary large, but finite index set $\Lambda$ is then defined by
$$
\hat f(h)=\sum_{\lambda \in \Lambda}f_\lambda(h) \psi_\lambda \text{ with }\hat f_\lambda(h):= h_\lambda X_\lambda,
$$
where $h=(h_\lambda)_{\lambda}\in \{0,1\}^\Lambda.\ $
%$\Lambda =\{(j,k),j\le J(\Sigma),k=1,\dots ,2^{jd}\}$.
The empirical version of $f$ is defined as
$$
\tilde f=\sum_{\lambda \in \Lambda}X_\lambda \psi_\lambda .
$$
%The risk is
%$$
%R(h,f):=R({\hat{f }}(h),f ) ={\bf E}_{f }\Vert
%{\hat{f}}(h)-f \Vert ^{2} ={\bf E}_{f }\sum_{\lambda \in \Lambda }
%({\hat{f}}_{\lambda}(h)-f _{\lambda})^{2}.
%$$
We write $\abs{h}=\#\{h_\lambda=1\}$ and $n=\#\Lambda$ for the cardinality of $\Lambda$.
Let us write $\Sigma_h$ for the covariance matrix of the $\xi_\lambda $ restricted to the indices
$\lambda $ for which $h_\lambda=1$, i.e.
$$
\Sigma_h=\diag(\sigma_\lambda ^2)_{\lambda \in \Lambda (h)}
%\Sigma_h=\left( \Sigma_{\lambda,\lambda'}\right)_{\lambda,\lambda'\in \Lambda (h)}
$$
with $\Lambda(h)=\{\lambda\,:\,h_\lambda=1\}$. By $\norm{A}$ we denote the operator norm, i.e. the largest absolute eigenvalue.

The random elements $(X_\lambda)_\lambda$ take values in the sample space ${\cal X}=\R^\Lambda$.
We now consider an arbitrary family ${\cal H}\subset{\cal H}_0:=\{h:{\cal X}\to\{0,1\}^\Lambda\}$ of Borel-measurable data-driven subset selection rules.
Define an estimator by minimizing in the family $\cal H$ the penalized empirical risk:
\begin{equation}
\label{hstar}
h^\star =\arg \min_{h\in {\cal H}} \left\{ \| \hat f(h)-\tilde f \|^2
+ 2Pen(h)
%-\|X\|^2
\right\},
\end{equation}
with the penalty
%[\MR: no $\Sigma$ in front of loglog!!!] [\LC Old but keep]
\begin{equation}\label{EqPen}
Pen (h)=2\sum_{j=1}^{\abs{h}}\sigma_{(j)_h}^2(\log(ne/j)+j^{-1}\log_+(n\norm{\Sigma})),
\end{equation}
where $\sigma_{(j)_h}^2$ denotes the $j$-th largest value among $\{h_\lambda\sigma_\lambda^2\}$ and $\log_+(z)=\max(\log z,0)$.
Remark that $h^\star$ is defined in an equivalent way by
$$
h^\star =\arg \min_{h\in {\cal H} } \bar R_{pen} (X,h),
$$
where
$$
\bar R_{pen} (X,h)= -\sum_{\lambda \in \Lambda } h_\lambda X_\lambda^2+ 2Pen(h).
$$
Then, define the data-driven estimator
\begin{equation}\label{EqEst}
f^\star = \sum_{\lambda \in \Lambda}h^\star_\lambda X_\lambda \psi_\lambda.
\end{equation}
The next lemma shows that one has an explicit risk hull, a concept introduced in full detail in \cite{CG}.

\begin{lemma}
\label{th:hull}
%For any $\alpha \ge 0$
The function
\begin{equation}
\label{hulldef}
\ell (f, h)= \sum_{\lambda \in \Lambda }(1-h_\lambda )f_\lambda^2 + Pen(h)+ \sqrt{2}\min\big(\tfrac1n,\norm{\Sigma}\big),
\end{equation}
with the penalty from \eqref{EqPen} is a risk hull, i.e. we have
\begin{equation}
\label{hulleq}
\mathbf{E} \sup_{h\in {\cal H}_0 } \left( \| \hat f(h)-f\|^2- \ell (f, h)\right) \le 0.
\end{equation}
\end{lemma}

\begin{proof}
%[(new proof for diagonal case)]
Recall $n=\#\Lambda$ and introduce the stochastic term
\begin{equation}
\label{eta}
\eta (h)= \sum_{\lambda \in \Lambda }h_\lambda \xi _{\lambda}^{2}.
\end{equation}
Remark that $\|\hat f(h)-f\|^2 =
\sum_{\lambda \in \Lambda }(1-h_\lambda)f_\lambda^2
+\eta(h)$ such that
\begin{equation}
\mathbf{E} \sup_{h\in {\cal H}_0 } \left( \| \hat f(h)-f\|^2-\sum_{\lambda \in \Lambda }(1-h_\lambda )f_\lambda^2  -Pen (h)- \sqrt{2}\min\big(\tfrac1n,\norm{\Sigma}\big) \right) \le 0
\end{equation}
follows from
\begin{equation}
\label{hu}
\mathbf{E} \sup_{h\in {\cal H}_0} \left( \eta (h) - Pen (h)\right)\le \sqrt{2}\min\big(\tfrac1n,\norm{\Sigma}\big).
\end{equation}
Let us write $\zeta_\lambda=\sigma_\lambda^{-1}\xi_\lambda\sim {\cal N}(0,1)$ and let $r_\lambda(h)$ denote the inverse rank of $h_\lambda\sigma_\lambda^2$ in
$(h_{\lambda'}\sigma_{\lambda'}^2)_{\lambda'}$ (e.g., $r_\lambda(h)=1$ if $h_\lambda\sigma_\lambda^2=\max_{\lambda'}h_{\lambda'}\sigma_{\lambda'}^2)$ such that
\[
\eta (h) - Pen(h)=\sum_{\lambda\in\Lambda}h_\lambda\sigma_\lambda^2\Big(\zeta_\lambda^2
- 2\Big(\log\Big(\frac{ne}{r_\lambda(h)}\Big)+r_\lambda(h)^{-1}\log_+(n\norm{\Sigma})\Big)\Big).
\]
Note that for any enumeration $(\lambda_j)_{j=1,\ldots,k}$ of $\{\lambda\,|\,h_\lambda=1\}$ by monotonicity:
\[
\sum_{\lambda\in\Lambda}h_\lambda\sigma_\lambda^2\Big(\log(ne/r_\lambda(h))
+r_\lambda(h)^{-1}\log_+(n\norm{\Sigma})\Big) \ge \sum_{j=1}^k\sigma_{\lambda_j}^2\Big(\log(ne/j)+j^{-1}
\log_+(n\norm{\Sigma})\Big)
\]
holds. We therefore obtain with the inverse order statistics $(\sigma^2_{(i)})$ and $(\zeta_{(i)}^2)$ (i.e. $\sigma^2_{(1)}\ge\sigma^2_{(2)}\ge\cdots$ etc.) of
$(\sigma_\lambda^2)_{\lambda\in\Lambda}$ and $(\zeta_\lambda^2)_{\lambda\in\Lambda}$, respectively,
\[
\E\Big[\sup_{h\in {\cal H}_0}\left(\eta (h) - Pen(h)\right)_+\Big]
\le \E\Big[\sum_{j=1}^n\sigma_{(j)}^2\Big(\zeta_{(j)}^2 - 2(\log(ne/j)+j^{-1}\log_+(n\norm{\Sigma}))\Big)_+ \Big].
\]
It remains to evaluate $\E[(\zeta_{(j)}^2 - 2(\log(ne/j)+j^{-1}\log_+(n\norm{\Sigma})))_+]$. We obtain by independence, $\log(\binom{n}{k})\le k\log(ne/k)$ and
by the Mill ratio inequality $P(\zeta_\lambda>t)\le t^{-1}e^{-t^2/2}$
\begin{align*}
P(\zeta_{(j)}^2>\kappa) &= P(\exists i_1,\ldots,i_j \forall l\in\{1,\ldots,j\}: \zeta_{i_l}^2>\kappa)\\
&\le \binom{n}{j}P(\zeta_\lambda^2>\kappa)^j \le \kappa^{-1/2}\exp(j\log(ne/j)-j\kappa/2).
\end{align*}
This implies for any $p>0$
\[ \E[(\zeta_{(j)}^2-p)_+]=\int_p^\infty P(\zeta_{(j)}^2>\kappa)\,d\kappa\le 2j^{-1}p^{-j/2}\exp(j\log(ne/j)-jp/2).
\]
We conclude
\[
\E\Big[\sum_{j=1}^n\sigma_{(j)}^2\left(\zeta_{(j)}^2 - 2(\log(ne/j)+j^{-1}\log_+(n\norm{\Sigma}))\right)_+\Big]
\]
\[
\le 2\norm{\Sigma}\sum_{j=1}^n j^{-1}(2\log(ne/j))^{-j/2}\exp(-\log_+(n\norm{\Sigma}))
\]
\[
\le \min\big(\tfrac1n,\norm{\Sigma}\big) \sup_n 2\sum_{j=1}^n j^{-1}(2\log(ne/j))^{-j/2}\le \sqrt{2}\min\big(\tfrac1n,\norm{\Sigma}\big),
\]
where $\sigma_{(j)}^2\le \norm{\Sigma}$ and the supremum is attained at $n=1$ with value $\sqrt{2}$.

\end{proof}

\smallskip

\begin{theorem}
\label{th:oracle}
Let $h^\star$ be the data-driven rule defined in (\ref{hstar}).
For any  $\delta \in(0,1)$,
we have
$$
\mathbf{E}_f\, \|\hat {f}(h^\star)-f\|^2\le
\left(1+\delta \right)\mathbf{E}_f\left[ \inf_{h\in {\cal H}}\Big(\sum_{\lambda\in\Lambda}(1-h_\lambda)f_\lambda^2-\sum_{\lambda\in\Lambda} h_\lambda(X_\lambda^2-f_\lambda^2)
+2Pen(h)\Big)\right]+\Omega_\delta,
$$
%\end{equation}
where
$$
\Omega_\delta:=4\sqrt{2}\min\big(\tfrac1n,\norm{\Sigma}\big) + \frac{2}{\delta}\sum_{\lambda \in \Lambda}\min(f_\lambda^2,\sigma^2_{\lambda}).
$$
\end{theorem}

\begin{proof}
%Let $\mu\in [0,\alpha-1]$.
In view of Lemma \ref{th:hull},
\begin{equation}
\ell(f,h)=\sum_{\lambda \in \Lambda } (1-h_\lambda )f_\lambda^2 + Pen(h)+\sqrt{2}\min\big(\tfrac1n,\norm{\Sigma}\big)
\end{equation}
is a risk hull, and therefore we have
\begin{equation}  \label{ep1}
\mathbf{E}_f \|\hat{f}(h^\star)-f \|^2\le
\mathbf{E}_f \ell(f,h^\star).
\end{equation}
On the other hand, since $h^\star$ minimizes $\bar{R}_{pen}(X,h)$ we
have
\begin{equation}
\label{ep2} \mathbf{E}_f \bar{R}_{pen}(X,h^\star) =\mathbf{E}_f\Big[\min_{h\in{\cal H}}
 \bar{R}_{pen}(X,h)\Big].
\end{equation}
In order to combine the inequalities (\ref{ep1}) and (\ref{ep2}), we rewrite
$\ell(f,h^\star)$ in terms of $\bar{R}_{pen}(X,h^\star )$
\begin{equation}
\label{l_mu}
\ell(f,h^\star)
=\bar{R}_{pen}(X,h^\star)+\|f\|^2+\sqrt{2}\min\big(\tfrac1n,\norm{\Sigma}\big)
+\sum_{\lambda \in \Lambda }h^\star_\lambda \xi_\lambda^2
+\sum_{\lambda \in \Lambda }2f_\lambda h^\star_\lambda \xi_\lambda + Pen(h^\star)-
2Pen(h^\star).
\end{equation}
Therefore, using this equation and (\ref{ep1}, \ref{ep2}), we obtain
\begin{equation}
\label{ep3}
\begin{split}
\mathbf{E}_f\|\hat {f}(h^\star )-f\|^2 \le&
\mathbf{E}_f\Big[\min_{h\in{\cal H}}
 \bar{R}_{pen}(X,h)\Big]+\norm{f}^2+ \sqrt{2}\min\big(\tfrac1n,\norm{\Sigma}\big) +2\mathbf{E}_f \,
\sum_{\lambda \in \Lambda }h_\lambda^\star f_\lambda \xi_\lambda \\
&+\mathbf{E}_f\,\biggl[\sum_{\lambda \in \Lambda }
h_\lambda^\star \xi_\lambda^2-Pen(h^\star)\biggr] .
\end{split}
\end{equation}
Remark now that for any deterministic index set $\Lambda'\subset\Lambda$
\begin{equation}
\label{comp}
\mathbf{E}_f \sum_{\lambda \in \Lambda' }2h_\lambda^\star f_\lambda \xi_\lambda
+\mathbf{E}_f \sum_{\lambda \in \Lambda' }2(1-h_\lambda^\star) f_\lambda \xi_\lambda
=\mathbf{E}_f \sum_{\lambda \in \Lambda' }2f_\lambda \xi_\lambda=0.
\end{equation}
This implies for $\Lambda_1:=\{\lambda \in \Lambda:f_\lambda^2 >
\sigma_{\lambda}^2 \}$
\begin{equation}
\label{I1}
\mathbf{E}_f \sum_{\lambda \in \Lambda }2h_\lambda^\star f_\lambda \xi_\lambda
=-\mathbf{E}_f \sum_{\lambda \in \Lambda_1}2(1-h_\lambda^\star) f_\lambda \xi_\lambda
+ \mathbf{E}_f \sum_{\lambda \in \Lambda_1^\complement}2h_\lambda^\star f_\lambda \xi_\lambda .
\end{equation}
Then, by the general inequality $2AB\le \frac{\delta}{2} A+\frac{2}{\delta}B$ for $A,B,\delta>0$ we obtain
\begin{equation}
\label{cauch}
\left|\mathbf{E}_f \sum_{\lambda \in \Lambda_1}2(1-h_\lambda^\star)\xi_\lambda f_\lambda \right|
\le \tfrac{\delta}{2} \mathbf{E}_f \sum_{\lambda \in \Lambda }(1-h_\lambda^\star)f_\lambda^2
+ \tfrac{2}{\delta} \mathbf{E}_f \sum_{\lambda \in \Lambda_1}(1-h_\lambda^\star)\xi_\lambda^2.
\end{equation}
Note that
\begin{equation}
\label{vari_1}
%\delta^{-1} \mathbf{E}_f \sum_{i\in I_1}(1-h_i^\star)\sigma_i\xi_i \le
\frac{2}{\delta}\mathbf{E}_f \sum_{\lambda \in \Lambda_1}(1-h_\lambda ^\star)
\xi_\lambda ^2 \le \tfrac{2}{\delta}\norm{\Sigma_{\Lambda_1}}_{tr}
\end{equation}
since $|1-h_\lambda^\star|\le 1$.
By (\ref{cauch}) and (\ref{vari_1}) we obtain
%for any $f \in f_0(\gamma_n)$
\begin{equation}
\label{var-Theta}
\left| \mathbf{E}_f \sum_{\lambda \in \Lambda_1}2(1-h_\lambda ^\star)f_\lambda \xi_\lambda \right|
\le   \tfrac{2}{\delta}\norm{\Sigma_{\Lambda_1}}_{tr} + \tfrac{\delta}{2} \mathbf{E}_f \sum_{\lambda \in \Lambda }(1-h_\lambda^\star)f_\lambda^2.
\end{equation}
In a similar way, we obtain
\begin{equation}
\label{cauch2}
\left|\mathbf{E}_f \sum_{\lambda \in \Lambda_1^\complement}2h_\lambda^\star\xi_\lambda f_\lambda \right|
\le \tfrac{\delta}{2} \mathbf{E}_f \sum_{\lambda \in \Lambda }h_\lambda^\star \xi_\lambda^2
+ \tfrac{2}{\delta} \mathbf{E}_f \sum_{\lambda \in \Lambda_1^\complement}h_\lambda^\star f_\lambda^2.
\end{equation}
Note that
\begin{equation}
\label{vari_2}
\tfrac{2}{\delta}\mathbf{E}_f \sum_{\lambda \in \Lambda_1^\complement}h_\lambda ^\star
f_\lambda ^2 \le \tfrac{2}{\delta} \sum_{\lambda \in \Lambda_1^\complement}f_\lambda ^2
\end{equation}
since $|h_\lambda^\star|\le 1$.
Using (\ref{cauch2}) and (\ref{vari_2}) one has
%for any $f \in f_0(\gamma_n)$
\begin{equation}
\label{var-Theta2}
\left| \mathbf{E}_f \sum_{\lambda \in \Lambda_1^\complement}2h_\lambda ^\star f_\lambda \xi_\lambda \right|
 \le \tfrac{2}{\delta} \sum_{\lambda \in \Lambda_1^\complement}f_\lambda ^2 + \tfrac{\delta}{2} \mathbf{E}_f \sum_{\lambda \in \Lambda }h_\lambda^\star \xi_\lambda^2.
\end{equation}
Note also that, since $h_\lambda \in \{0,1\}$, we have
$$
\mathbf{E}_f\|\hat {f}(h^\star )-f\|^2 =\mathbf{E}_f
\sum_{\lambda \in \Lambda }(1-h_\lambda^\star)f_\lambda^2
+\mathbf{E}_f \sum_{\lambda \in \Lambda }h^\star_\lambda\xi_\lambda^2.
$$
Insertion of (\ref{var-Theta}) and (\ref{var-Theta2}) into \eqref{I1} yields
\begin{equation}
\label{cross}
\left| \mathbf{E}_f \sum_{\lambda \in \Lambda }2h_\lambda ^\star f_\lambda \xi_\lambda \right|
\le \tfrac{\delta}{2} \mathbf{E}_f\|\hat {f}(h^\star )-f\|^2
+ \tfrac{2}{\delta}\norm{\Sigma_{\Lambda_1}}_{tr}+ \tfrac{2}{\delta} \sum_{\lambda \in \Lambda_1^\complement}f_\lambda ^2.
\end{equation}
By using the risk hull as in Lemma \ref{th:hull}, one obtains
\begin{equation}
\label{hull2}
\mathbf{E}_f\,\biggl[\sum_{\lambda \in \Lambda}h_\lambda^\star\xi_\lambda^2-Pen(h^\star)\biggr]
\le \sqrt{2}\min\big(\tfrac1n,\norm{\Sigma}\big).
\end{equation}
Inserting (\ref{var-Theta}), (\ref{var-Theta2}) and (\ref{hull2}) into (\ref{ep3}) yields
\begin{align}
\label{fin1}
\mathbf{E}_f\|\hat {f}(h^\star )-f\|^2& \le \mathbf{E}_f\Big[\min_{h\in{\cal H}}
 \bar{R}_{pen}(X,h)\Big]+\norm{f}^2+ \sqrt{2}\min\big(\tfrac1n,\norm{\Sigma}\big)+ \tfrac{2}{\delta}\sum_{\lambda \in \Lambda}\min(f_\lambda^2,\sigma^2_{\lambda})\\
 &\quad
+\sqrt{2}\min\big(\tfrac1n,\norm{\Sigma}\big)+ \frac{\delta}{2} \mathbf{E}_f\|\hat {f}(h^\star )-f\|^2.
\end{align}
Using (\ref{fin1}) we obtain,
%by choosing $\mu =(\alpha -1)/2$ [\LC Check,$\mu,\alpha $],
$$
(1-\tfrac{\delta}{2} )\mathbf{E}_f\|\hat {f}(h^\star )-f\|^2 \le
\mathbf{E}_f\Big[\min_{h\in{\cal H}}
 \bar{R}_{pen}(X,h)+\norm{f}^2\Big]+ 2\sqrt{2}\min\big(\tfrac1n,\norm{\Sigma}\big)+\frac{2}{\delta}\sum_{\lambda \in \Lambda}\min(f_\lambda^2,\sigma^2_{\lambda}).
$$
Finally, we let the bias explicitly appear in
\[  \bar{R}_{pen}(X,h)+\norm{f}^2=\sum_{\lambda\in\Lambda}(1-h_\lambda)f_\lambda^2-\sum_{\lambda\in\Lambda} h_\lambda(X_\lambda^2-f_\lambda^2)+2Pen(h)
\]
and the result follows from $(1-\frac{\delta}{2})^{-1}\le 1+\delta$ for $\delta\in[0,1]$.
\end{proof}

\section{Sparse representations}
\label{sec:spa}

Let us consider the intuitive version of sparsity by assuming a small proportion of nonzero coefficients (cf. \cite{ABDJ}),
i.e. the family
$$
{\cal F}_0(\gamma_n):=\Big\{f:\sum_{\lambda\in\Lambda}{\bf 1}(f_\lambda\not=0)\le n\gamma_n\Big\}
$$
where $\gamma_n:=\#\{\lambda\in\Lambda\,|\,f_\lambda\not=0\}/n$ denotes the maximal proportion of nonzero coefficients.

Throughout, we assume that this proportion $\gamma_n$ is such that asymptotically
$$\gamma_n\to 0 \text{ and }n\gamma_n\to \infty.$$

\subsection{Full subset selection}

The goal here is to study the accuracy of the full model selection over the whole family of estimators. Each coefficient may be chosen to be inside
or outside the model.
Let us consider the case where ${\cal  H}$ denotes all
deterministic subset selections,
\begin{equation}
\label{Hms}
{\cal  H}=\{ h:{\cal X}\to \{0,1\}^\Lambda\,|\, h(x)={\bf 1}_{\Lambda'},\,\Lambda'\subset\Lambda\}.
\end{equation}

\begin{theorem}
\label{th:ms}
Let $h^\star$ be the data-driven rule defined in (\ref{hstar}) with ${\cal  H}$ as in (\ref{Hms}).
We have, for $n\to\infty$, uniformly over $f\in {\cal F}_0(\gamma_n)$,
\begin{equation}
\label{bound1}
\mathbf{E}_f\, \|\hat {f}(h^\star)-f\|^2\le(4+o(1))\norm{\Sigma_{h^f}}\left( n\gamma_n\log(\gamma_n^{-1})
+\log(n\gamma_n)\log_+(n\norm{\Sigma})\right)+4\sqrt{2}\min\big(\tfrac1n,\norm{\Sigma}\big).
\end{equation}

In particular, if $\log_+(\norm{\Sigma})=O(\log n)$ (i.e., any polynomial growth for $\norm{\Sigma}$ is admissible)
and $\frac{\norm{\Sigma_{h^f}}}{\norm{\Sigma}}\max(n\norm{\Sigma},1)n\gamma_n\log(\gamma_n^{-1})\to \infty $, then we obtain
\begin{equation}
\label{bound1b}
\mathbf{E}_f\, \|\hat {f}(h^\star)-f\|^2\le(4+o(1))\norm{\Sigma_{h^f}}n\gamma_n\log(\gamma_n^{-1}).
%+\log(n\gamma_n)\log(n\norm{\Sigma})\right).
\end{equation}
\end{theorem}

\begin{proof}
For $f\in {\cal F}_0(\gamma_n)$ the right-hand side in Theorem \ref{th:oracle} can be bounded
by considering the oracle $h^f={\bf 1}(\{\lambda\,:\,f_\lambda\not=0\})$ such that
\begin{align}
(1+\delta)\mathbf{E}_f\left[ \Big(-\sum_{\lambda\in\Lambda} h_\lambda^f(X_\lambda^2-f_\lambda^2)+2Pen(h^f)\Big)\right]+\Omega_\delta
&\le (1+\delta)2Pen(h^f)+\Omega_\delta.
\end{align}
We will use the following inequality, as $J\to \infty $,
\begin{equation}
\label{sumlog}
\sum_{j=1}^J(\log(ne/j)+j^{-1}\log_+(n\norm{\Sigma}))\le (J\log(ne/J)+\log(J)\log_+(n\norm{\Sigma}))(1+o(1)),
\end{equation}
by comparison with the integral.
Since $\abs{h^f}\le n\gamma_n$, we obtain that
\[
Pen(h^f)\le 2\norm{\Sigma_{h^f}}
\left( \sum_{j=1}^{\abs{h^f}}(\log(ne/j)+j^{-1}\log_+(n\norm{\Sigma}))\right)
\]
\[
\le 2\norm{\Sigma_{h^f}}\left( n\gamma_n\log(\gamma_n^{-1})+\log(n\gamma_n)\log_+(n\norm{\Sigma})\right)(1+o(1)),
\]
as $n\to\infty$.
%[\LC Check].
On the other hand, we have
\[
\Omega_\delta=4\sqrt{2}\min\big(\tfrac1n,\norm{\Sigma}\big)+\frac{2}{\delta}
\sum_{\lambda:f_\lambda\not=0}\min(\sigma^2_{\lambda},f_\lambda^2).
\]
We use $\sum_{\lambda:f_\lambda\not=0} \sigma^2_{\lambda}\le n\gamma_n\norm{\Sigma_{h^f}}$ which shows
$$
\Omega_\delta \le \frac{4\sqrt{2}}{n}+\frac{2}{\delta}n\gamma_n\norm{\Sigma_{h^f}}.
$$
Choosing $\delta\to 0$ such that $\delta^{-1}=o(\log(\gamma_n^{-1}))$, e.g. $\delta=1/\log\log(\gamma_n^{-1})$, we thus find, as $n\to \infty $,
\begin{equation}
\label{Omega_del}
%\Omega_\delta
\frac{2}{\delta}n\gamma_n\norm{\Sigma_{h^f}}= o\Big( \norm{\Sigma_{h^f}}n\gamma_n\log(\gamma_n^{-1})\Big).
\end{equation}
Using Theorem \ref{th:oracle}, Equation (\ref{Omega_del}) we have (\ref{bound1}).
Moreover, using the bounds on $\norm{\Sigma_{h^f}}$ and $\norm{\Sigma}$ we obtain (\ref{bound1b}).
\end{proof}
%\[
%\mathbf{E}_f\, \|\hat {f}(h^\star)-\theta
%\|^2\le(4+o(1))\norm{\Sigma_{h^f}}n\gamma_n\log(\gamma_n^{-1}).
%\]
%[\LC In fact it is (as in next section)]
%\MR: discuss direct case, $\alpha=0$, Golubev etc.

\subsection{Threshold estimators}

Consider now a  family of threshold estimators. The problem is to study the data-driven selection of the threshold.
Let us consider the case where ${\cal  H}$ denotes the threshold selection rules with arbitrary threshold values $t>0$
\begin{equation}
\label{Htr}
{\cal H}=\{h((X_\lambda)_\lambda)={\bf 1}(\lambda\,:\,\abs{X_\lambda}> \sigma_\lambda t) \,|\,t>0\}.
\end{equation}
Note that $\cal H$ consists of $n=\#\Lambda$ different subset selection rules only and can be implemented efficiently using the order statistics of $(\abs{X_\lambda}/\sigma_\lambda)_\lambda$.

\begin{theorem}
\label{th:tr}
Let $h^\star$ be the data-driven rules defined in (\ref{hstar}) with ${\cal  H}$ as in (\ref{Htr}).
%Assume $\log(n\norm{\Sigma})=o(n\gamma_n\norm{\Sigma_{h_f}}/\log (n\gamma_n\norm{\Sigma_{h_f}}))$.
If $\norm{\Sigma_{h_f}}\log(\gamma_n^{-1})\to\infty$, then we have, for $n\to\infty$, uniformly over $f\in {\cal F}_0(\gamma_n)$
\begin{align}
\label{bound2}
\mathbf{E}_f\, \|\hat {f}(h^\star)-f\|^2 \le &
\Big(4n\gamma_n(\norm{\Sigma_{h_f}}\log(\gamma_n^{-1})+8\norm{\Sigma}\gamma_n(\log(\gamma_n^{-1}))^{1/2})\\       &+2\log_+(n\norm{\Sigma})(2\norm{\Sigma_{h_f}}\log(n\gamma_n)
+4\norm{\Sigma} \log_+(n\gamma_n^{2}))\Big)(1+o(1)).
\end{align}
Assuming for $\Sigma$ the growth bounds
\[\norm{\Sigma}=O(\norm{\Sigma_{h_f}}\gamma_n^{-1})\text{ and } \norm{\Sigma}\log_+(n\norm{\Sigma})=o(\norm{\Sigma_{h_f}}n\gamma_n
\log(\gamma_n^{-1})/\log_+(n\gamma_n^2)),
\]
with a second condition always checked if $\log_+(n\gamma_n^2)=0$, this inequality simplifies to
\[ \mathbf{E}_f\, \|\hat {f}(h^\star)-f\|^2\le (4+o(1))\norm{\Sigma_{h^f}} n\gamma_n\log(\gamma_n^{-1}).
\]
\end{theorem}

\begin{proof}
Let us now evaluate the right-hand side of the oracle inequality in Theorem \ref{th:oracle} for the threshold selection rules
with arbitrary threshold values $t>0$ defined in (\ref{Htr}).
Given an oracle parameter $t^0>1$ (to be determined below), we set
 $\tau_\lambda:=\sigma_\lambda t^0$. We obtain with $R_\lambda$ denoting the (inverse) rank of the coefficient with index $\lambda$ among $(\sigma_\lambda^2 1(\abs{X_\lambda}>\tau_\lambda))_{\lambda\in\Lambda}$
%$h^0((X_\lambda)_\lambda)={\bf 1}
%(\lambda\,:\,\abs{X_\lambda}> t^0\sigma_\lambda)$ and $\sup_{h\in\{0,1\}^\Lambda}\norm{\Sigma_h}=\norm{\Sigma}$
%[\LC Check Pen]
\begin{align}\label{EqThrRisk}
&\mathbf{E}_f\left[ \inf_{h\in {\cal H}}\Big(\sum_{\lambda\in\Lambda}(1-h_\lambda)f_\lambda^2-\sum_{\lambda\in\Lambda} h_\lambda(X_\lambda^2-f_\lambda^2)+2Pen(h)\Big)\right]\\
&\le \mathbf{E}_f\Big[ \sum_{\lambda\in\Lambda}\Big({\bf 1}(\abs{X_\lambda}\le\tau_\lambda)f_\lambda^2- {\bf 1}(\abs{X_\lambda}>\tau_\lambda) (X_\lambda^2-f_\lambda^2)\\
&\qquad\qquad
+4\sigma_\lambda^2 {\bf 1}(\abs{X_\lambda}>\tau_\lambda)(\log(en/R_\lambda)
+R_\lambda^{-1}\log_+(n\norm{\Sigma}))\Big)\Big].
\end{align}

%&\le \sum_{\lambda\in\Lambda}\Big(P(\abs{X_\lambda}\le\sigma_\lambda t^0)f_\lambda^2+4\norm{\Sigma}P(\abs{X_\lambda}>\sigma_\lambda t^0)
%\log(en/\E\abs{h^0})(1+o(1))\\
%&\qquad -\E[{\bf 1}(\abs{X_\lambda}>\sigma_\lambda t^0) (X_\lambda^2-f_\lambda^2)]\Big),
%\end{align*}
%where in the last line $\E[-\abs{h^0}\log(\abs{h^0})]\le -\E[\abs{h^0}]\log(\E\abs{h^0})$ by Jensen's inequality was used.

Let us first show that $\E_f[{\bf 1}(\abs{X_\lambda}>\tau_\lambda) (X_\lambda^2-f_\lambda^2)]$ is always non-negative.
By symmetry $X_\lambda':=f_\lambda-\xi_\lambda$ has the same law as $X_\lambda$.
Defining the function $g(\xi):={\bf 1}(\abs{f_\lambda+\xi}>\tau_\lambda)((f_\lambda+\xi)^2-f_\lambda^2)$,
we check by considering the different cases that $g(\xi)+g(-\xi)\ge 0$ holds. We conclude
\begin{align*}
\E_f[{\bf 1}(\abs{X_\lambda}>\tau_\lambda) (X_\lambda^2-f_\lambda^2)] &=
\tfrac12
\E_f[g(\xi_\lambda)+g(-\xi_\lambda)]\ge 0.
\end{align*}
Hence, the term with a minus sign in \eqref{EqThrRisk} can be discarded for an upper bound.

Let us now  consider the coefficients that contain a signal part (i.e. with $f_\lambda\not=0$).
The following inequality will be helpful to obtain a bound independent of the size of $\abs{f_\lambda}$.
Let us denote by $r_\lambda^f$ the corresponding inverse rank within  $(\sigma_\lambda^2 {\bf 1}(f_\lambda\not=0))_{\lambda\in\Lambda}$.
With $f_\lambda^2\le (\abs{\xi_\lambda}+\tau_\lambda)^2$ on the event $\{\abs{X_\lambda}\le\tau_\lambda\}$ we obtain
\begin{align}\label{sig}
&\sum_{\lambda\in\Lambda,f_\lambda\not=0}\Big({\bf 1}(\abs{X_\lambda}\le\tau_\lambda)f_\lambda^2
+4\sigma_\lambda^2 {\bf 1}(\abs{X_\lambda}>\tau_\lambda)(\log(en/R_\lambda)+R_\lambda^{-1}\log_+(n\norm{\Sigma}))\Big)\\
&\le \sum_{\lambda\in\Lambda,f_\lambda\not=0}\max\Big((\abs{\xi_\lambda}+\tau_\lambda)^2,4\sigma_\lambda^2 (\log(en/R_\lambda)+R_\lambda^{-1}\log_+(n\norm{\Sigma}))\Big)\\
&\le \sum_{\lambda\in\Lambda,f_\lambda\not=0}\max\Big((\abs{\xi_\lambda}+\tau_\lambda)^2,4\sigma_\lambda^2 (\log(en/r_\lambda^f)+(r_\lambda^f)^{-1}\log_+(n\norm{\Sigma}))\Big),
\end{align}
where for the last inequality we have used that for $n\gamma_n$ distinct values $R_\lambda\in\N$ the expression is maximal in the case $R_\lambda=r_\lambda^f$.

The general identity ${\bf E}[\max(Z,c)]=c+\int_c^\infty P(Z\ge z)dz$ applied to $Z=(\abs{\xi_\lambda}+\tau_\lambda)^2$ and deterministic $c_\lambda\ge\tau_\lambda^2$ yields
\begin{align}\label{Emax}
{\bf E}[\max((\abs{\xi_\lambda}+\tau_\lambda)^2,c_\lambda)]&\le c_\lambda+\int_{c_\lambda}^\infty P(\abs{\xi_\lambda}\ge\sqrt{z}-\tau_\lambda)\,dz
\le c_\lambda+2e^{-(\sqrt{c_\lambda}-\tau_\lambda)^2/(2\sigma_\lambda^2)}.
\end{align}

%Since $(\abs{\xi_\lambda}+\tau_\lambda)^2/\tau_\lambda^2\downarrow 1$ for any realisation of $\xi_\lambda$ as $n\to\infty$, monotone convergence [\MR check!] shows
%\begin{align*}
%&{\bf E}\Big[\max\Big((\abs{\xi_\lambda}+\tau_\lambda)^2,4\sigma_\lambda^2 (\log(en/r_\lambda^f)+(r_\lambda^f)^{-1}\log(n\norm{\Sigma}))\Big)\Big]\\
%&=\tau_\lambda^2
%\max\Big(1,4\sigma_\lambda^2 (\log(en/r_\lambda^f)+(r_\lambda^f)^{-1}\log(n\norm{\Sigma}))\tau_\lambda^{-2}\Big)(1+o(1))
%=4 (\beta+1)\log(en/r_\lambda^f)(1+o(1)).
%\end{align*}

In order to ensure $\tau_\lambda^2\le c_\lambda:=4\sigma_\lambda^2 (\log(en/r_\lambda^f)+(r_\lambda^f)^{-1}\log_+(n\norm{\Sigma}))$ whenever $f_\lambda\not=0$, we are lead to choose
\begin{equation}\label{Eqt0restr}
t^0= \sqrt{4\log(e/\gamma_n)}.
\end{equation}
In the sequel we bound $\sigma_\lambda^2$ simply by $\norm{\Sigma_{h_f}}$ in the case $f_\lambda\not=0$.
Then using again the bound on sums of logarithms (\ref{sumlog}) and $\#\{f_\lambda\not=0\}\le n\gamma_n$ as well as the concavity of $e^{-x}$ for bounding the sum of exponentials,
we obtain  that \eqref{EqThrRisk} over the signal part satisfies
\begin{align}
&\E_f\Big[\sum_{\lambda\in\Lambda,f_\lambda\not=0}\Big({\bf 1}(\abs{X_\lambda}\le\tau_\lambda)f_\lambda^2
+4\sigma_\lambda^2 {\bf 1}(\abs{X_\lambda}>\tau_\lambda)(\log(en/R_\lambda)+R_\lambda^{-1}\log_+(n\norm{\Sigma}))\Big)\Big]\\
&\le \sum_{\lambda\in\Lambda,f_\lambda\not=0}(c_\lambda+2e^{-(\sqrt{c_\lambda}-\tau_\lambda)^2/(2\sigma_\lambda^2)})
\le n\gamma_n(C_n\norm{\Sigma_{h_f}}+2 e^{-(C_n-(t^0)^2)/2}),\,
\end{align}
where
\begin{equation}\label{C_n}
C_n=(4+o(1))(\log(\gamma_n^{-1})+\log_+(n\norm{\Sigma})\log(n\gamma_n)/(n\gamma_n)).
\end{equation}

Owing to $C_n\norm{\Sigma_{h_f}}\to\infty$ we even have
\begin{align}
&\E_f\Big[\sum_{\lambda\in\Lambda,f_\lambda\not=0}\Big({\bf 1}(\abs{X_\lambda}\le\tau_\lambda)f_\lambda^2
+4\sigma_\lambda^2 {\bf 1}(\abs{X_\lambda}>\tau_\lambda)(\log(en/R_\lambda)+R_\lambda^{-1}\log_+(n\norm{\Sigma}))\Big)\Big]
\nonumber\\
&\le \norm{\Sigma_{h_f}}n\gamma_n C_n(1+o(1)).\label{sig2}
\end{align}

On the other hand, for the non-signal part $f_\lambda=0$, we introduce $N_\tau:=\sum_{\lambda\in\Lambda}{\bf 1}(\abs{\xi_\lambda}>\tau_\lambda)$
and we use the large deviation bound:
\[
{\bf E}[N_\tau]=n P(\abs{\xi_\lambda}>\tau_\lambda)\le 2n(t^0)^{-1}e^{-(t^0)^2/2}.
%=O(\gamma_n^{(\beta+1)}/\sqrt{\log(\gamma_n^{-1})}).
\]
%Let us denote with $R_\lambda^0$ the inverse rank within $(\sigma_\lambda^2 1(\abs{\xi_\lambda}>\tau_\lambda))_{\lambda:f_\lambda=0}$
Again by considering worst case permutations instead of the ranks, using (\ref{sumlog}) and by Jensen's inequality for the concave functions $\log(x),x\log(en/x)$ we infer:
\begin{align}\label{Eq42}
&\mathbf{E}_f\left[\sum_{\lambda:f_\lambda=0}\Big({\bf 1}(\abs{X_\lambda}\le\tau_\lambda)f_\lambda^2+4\sigma_\lambda^2 {\bf 1}(\abs{X_\lambda}>\tau_\lambda)(\log(en/R_\lambda)+R_\lambda^{-1}\log_+(n\norm{\Sigma}))\Big)\right]\\
&\le 4\norm{\Sigma}\mathbf{E}_f\left[\sum_{\lambda\in\Lambda} {\bf 1}(\abs{\xi_\lambda}>\tau_\lambda)(\log(en/R_\lambda)+R_\lambda^{-1}\log_+(n\norm{\Sigma}))\right]\\
&\le 4\norm{\Sigma}\mathbf{E}\left[\sum_{j=1}^{N_\tau}(\log(en/j)+j^{-1}\log_+(n\norm{\Sigma}))\right]\\
&\le 4\norm{\Sigma}\mathbf{E}\left[(N_\tau\log(en/N_\tau)+\log(N_\tau)\log_+(n\norm{\Sigma}))\right] (1+o(1))\\
&\le 4\norm{\Sigma} (2n (t^0)^{-1} e^{-(t^0)^2/2}(1+t_0^2/2)+(\log n-(t^0)^2/2)\log_+(n\norm{\Sigma}))(1+o(1))\\
&\le  2\norm{\Sigma}(2n e^{-(t^0)^2/2}t^0+(2\log n-(t^0)^2)\log_+(n\norm{\Sigma}))(1+o(1)).
\end{align}
For the $t^0$ chosen, the total bound over \eqref{EqThrRisk} is thus, by (\ref{sig2}), (\ref{Eq42}) and by definition of $C_n$ in (\ref{C_n}),
\begin{align}
& n\gamma_n(1+o(1))\Big(\norm{\Sigma_{h_f}}C_n+2\norm{\Sigma} (2e^{-(t^0)^2/2}t^0+(2\log n-(t^0)^2)\log_+(n\norm{\Sigma})/(n\gamma_n))\Big)\\
&=n\gamma_n(1+o(1))\Big(4\norm{\Sigma_{h_f}}(\log(\gamma_n^{-1})+\log_+(n\norm{\Sigma})\log(n\gamma_n)/(n\gamma_n))\\
&\quad+2\norm{\Sigma} (4\gamma_n\sqrt{\log(\gamma_n^{-1})}+2\log(n\gamma_n^{2})\log_+(n\norm{\Sigma})/(n\gamma_n))\Big).
\end{align}
This yields the asserted general bound and inserting the bound for $\log_+(n\norm{\Sigma})$ gives directly the second bound.
\end{proof}

\subsection{Discussion}
%[later\MR]
%If we assume the sparsity constraint $\abs{\Lambda_1}=n\gamma_n$ with $n\gamma_n\to \infty$ and $\gamma_n\to 0$ and the diagonal but heterogeneous
%case $\Sigma=\diag(\sigma_\lambda ^2)$, then our bound on the right-hand side is $4\norm{\Sigmax}n\gamma_n\log(\gamma_n^{-1})(1+o(1)).$
%+\frac{\sqrt{2}}{(\alpha-1)n}+2\delta^{-1}n\gamma_n\Big)(1+o(1))
%\]
%and by choosing $\delta\to 0$ such that $\delta^{-1}=o(\log(\gamma_n^{-1}))$, this evaluates to $ 4\sigma^2(n\gamma_n\log(\gamma_n^{-1})(1+o(1))$,
%which is only twofold the minimax risk, while Golubev obtains fourfold the minimax risk for his simple approach. [\LC Not now in Golubev (2011)!!]
{\it Heterogeneous case.}
One may compare the method and its accuracy with other results in related frameworks. For example, \cite{R} considers a very close framework
of model selection in inverse problems by using the SVD approach. This results in a noise $(\xi_\lambda)$ which is heterogeneous and diagonal.
\cite{J,JP} study the related topic of inverse problems and Wavelet Vaguelette Decomposition (WVD), built on \cite{BM}.
The framework in \cite{J} is more general than ours. However, this leads to less precise results. In all their results \cite{JP,R}, there exist
universal constants which are not really controlled.
This is even more important for the constants inside the method, for example in the penalty.
Our method contains an explicit penalty. It is used in the mathematical results and also in simulations without additional tuning.
 A possible extension of our method to the dependent WVD case does not seem straight-forward.

\smallskip \noindent
{\it Homogeneous case.}
Let us compare with other work for the homogeneous setting $\Sigma = \sigma^2Id$.
There exist a lot of results in this framework, see e.g. \cite{ABDJ,J,M,WZ}.
Again those results contain universal constants, not only in the mathematical results, but even inside the methods.
For example, constants in front of the penalty, but also inside the FDR technique, with an hyper-parameter $q_n$ which has to be tuned.

The perhaps closest paper to our work is \cite{G11} in the homogeneous case. Our penalty is analogous to ``twice the optimal'' penalty considered in \cite{G11}.
This is due to difficulties in the heterogenous case, where the stochastic process that one needs to control is much more involved
in this setting. Indeed, there is no more symmetry inside this stochastic process, since each empirical coefficient has its own variance.
The problem and the penalty do not only depend on the number of coefficients that one selects, but also on their position.

This leads to a result $4\norm{\Sigma}n\gamma_n\log(\gamma_n^{-1})$, where one gets a constant $2\sigma^2n\gamma_n\log(\gamma_n^{-1})$ in \cite{G11}.
The potential loss of the factor 2 in the heterogeneous framework might possibly be avoidable in theory, but in simulations the results seem comparably less sensitive to this factor than to other modifications, e.g. to how many data points, among the $n\gamma_n$ non-zero coefficients, are close to the critical threshold level, which defines some kind of effective sparsity of the problem (often muss less than $n\gamma_n$). This effect is not treated in the theoretical setup in all of the FDR-related studies, where implicitly a worst case scenario of the coefficients' magnitude is understood.

\section{Minimax lower bound}
\label{sec:low}

\begin{theorem}
\label{th:lower}
For any estimator $\hat f_n$ based on $n$ observations we have the minimax lower bound
\[ \sup_{f\in{\cal F}_0(\gamma_n)}{\bf E}_f[\norm{\hat{f}_n-f}^2]
\ge \sup_{\alpha_n\in S_\Lambda(n\gamma_n,c_n)} 2\big(1+o(1)\big)\Big(\sum_{\lambda\in\Lambda} \sigma_\lambda ^2\alpha_{\lambda,n}\log(\alpha_{\lambda,n}^{-1})\Big)
\]
for some $c_n\to 0$ where $S_\Lambda(R,c)=\{\alpha\in[0,c]^\Lambda\,|\,\sum_\lambda\alpha_\lambda\le R(1-c)\}$ denotes the intersection of $c$-times the $n$-dimensional
unit cube with $R(1-c)$-times the $n$-simplex and where $o(1)\to 0$ as $n\to\infty$.

Distributing mass uniformly over the $r_n$ indices with largest values $\sigma_\lambda$  yields the lower bound, as $n\to\infty$,
\[ \sup_{f\in{\cal F}_0(\gamma_n)}{\bf E}_f[\norm{\hat{f}_n-f}^2]
\ge 2n\gamma_n\log(\gamma_n^{-1})\big(1+o(1)\big)\frac 1{r_n}\sum_{i=1}^{r_n} \sigma_{(i)} ^2
\]
in terms of the inverse order statistics $\sigma_{(i)}^2$, provided $\log(n/r_n)=o(\log(\gamma_n^{-1}))$ (i.e., $r_n$ must be somewhat larger than $n\gamma_n$).

Note that for polynomial growth $\sigma_{(i)}^2\sim (n-i)^\beta$, $\beta>0$,
%and for $r_n=o(n)$, we  have $\sigma_{(r_n)}^2/\sigma_{(1)}^2\to 1$ and
the lower bound is, as $n\to\infty$,
\[ \sup_{f\in{\cal F}_0(\gamma_n)}{\bf E}_f[\norm{\hat{f}_n-f}^2]
\ge 2\big(1+o(1)\big)\norm{\Sigma}n\gamma_n\log(\gamma_n^{-1}).\]

\end{theorem}

\begin{remark}
The lower bound is a kind of weighted entropy.
In contrast to the upper bounds above the minimax (and the Bayes) lower bound does not involve the quantity $\norm{\Sigma_{h_f}}$, individual to each unknown $f$.
In the proof for this heterogeneous model, conceptually we need to allow for a high complexity of the class ${\cal F}_0(\gamma_n)$,
leading to the entropy factor $\log(\gamma_n^{-1})$, and to put more prior probability on coefficients with larger variance, which explains the abstract weighted entropy expression.
\end{remark}

\begin{proof}

Consider for each coefficient $f_\lambda $ the following Bayesian
prior, which turns out to be asymptotically least favorable:
\[ \pi_\lambda =(1-\alpha_{\lambda ,n})\delta_0+\alpha_{\lambda ,n}\delta_{\mu_{\lambda ,n}},\quad
\lambda\in\Lambda,
\]
with some $\mu_{\lambda,n }\ge 0$. Without loss of generality we may assume $c_n\downarrow 0$ so slowly that $c_n\sqrt{n\gamma_n}\to\infty$.
Introducing the number of non-zero entries $N:=\sum_\lambda {\bf 1}(f_\lambda \not=0)$ and writing $P$ for the joint law of prior and observations, we deduce by Chebyshev inequality
\[ P(f
\notin {\cal F}_0(\gamma_n))=P(N> n\gamma_n)=P(N-n\gamma_n(1-c_n)>n\gamma_nc_n)
\le \frac{\Var(N)}{(c_nn\gamma_n)^2}\le\frac{n\gamma_n}{(c_nn\gamma_n)^2}\to 0.
\]
The property $P(f\in {\cal F}_0(\gamma_n))\to 1$ then implies that the Bayes-optimal risk, derived below, will be an asymptotic minimax lower bound over ${\cal F}_0(\gamma_n)$.

We need to calculate the Bayes risk and find the posterior law of $f_\lambda \in\{0,\mu_{\lambda ,n}\}$ for each coordinate $\lambda $:
\[ P(f_\lambda =\mu_{\lambda ,n}|X_\lambda =x)=\frac{\alpha_{\lambda ,n}\phi_{\mu_{\lambda ,n},\sigma_\lambda ^2}(x)}
{(1-\alpha_{\lambda ,n})\phi_{0,\sigma_\lambda ^2}(x)+\alpha_{\lambda ,n}\phi_{\mu_{\lambda ,n},\sigma_\lambda ^2}(x)}.
\]
Since we deal with quadratic loss, the Bayes estimator $\hat f_\lambda $ equals the conditional expectation $\E[f_\lambda |X_\lambda ]$
 and the Bayes risk the expectation of the conditional variance, which is calculated as
\begin{equation}
 \E[\Var(f_\lambda \,|\,X_\lambda )]=\E[f_\lambda ^2]-\E[\E[f_\lambda |X_\lambda ]^2]
=\mu_{\lambda ,n}^2\Big(\alpha_{\lambda ,n}-\int \frac{\alpha_{\lambda ,n}^2\phi_{\mu_{\lambda ,n},\sigma_\lambda ^2}(x)^2}
{(1-\alpha_{\lambda ,n})\phi_{0,\sigma_\lambda ^2}(x)+ \alpha_{\lambda ,n}\phi_{\mu_{\lambda ,n},\sigma_\lambda ^2}(x)}\,dx\Big).
\end{equation}
The integral can be transformed into an expectation with respect to $Z\sim {\cal N}(0,1)$ and bounded by Jensen's inequality:
\begin{align*}
&\int \frac{\alpha_{\lambda ,n}^2\phi_{\mu_{\lambda ,n},\sigma_\lambda ^2}(x)^2} {(1-\alpha_{\lambda ,n})\phi_{0,\sigma_\lambda ^2}(x)+ \alpha_{\lambda ,n}\phi_{\mu_{\lambda ,n},\sigma_\lambda ^2}(x)}\,dx\\
&\qquad =\alpha_{\lambda ,n}\E\Big[\Big(1+\alpha_{\lambda ,n}^{-1}(1-\alpha_{\lambda ,n})\exp(\sigma_\lambda ^{-1}Z-\mu_{\lambda ,n}^2/(2\sigma_\lambda ^2))\Big)^{-1} \Big]\\
&\qquad \le \alpha_{\lambda ,n} \Big(1+\alpha_{\lambda ,n}^{-1}(1-\alpha_{\lambda ,n})\E[\exp(\sigma_\lambda ^{-1}Z-\mu_{\lambda ,n}^2/(2\sigma_\lambda ^2))]\Big)^{-1} \\
&\qquad =\alpha_{\lambda ,n}
\Big(1+\alpha_{\lambda ,n}^{-1}(1-\alpha_{\lambda ,n})\exp((1-\mu_{\lambda ,n}^2)/(2\sigma_\lambda ^2))\Big)^{-1}.
\end{align*}

Since $\alpha_{\lambda,n}\to 0$ uniformly, we just select
\[\mu_{\lambda ,n}=\sigma_\lambda \sqrt{2(1-(\log c_n^{-1})^{-1/2})\log(\alpha_{\lambda ,n}^{-1})}\]
such that
\[\E[\Var(f_\lambda \,|\,X_\lambda )]\ge 2\sigma_\lambda ^2\alpha_{\lambda ,n}(1-(\log c_n^{-1})^{-1/2})\log(\alpha_{\lambda ,n}^{-1})
(1-((1+(1-\alpha_{\lambda ,n})\alpha_{\lambda ,n}^{-(\log c_n^{-1})^{-1/2}}e^{1/(2\sigma_\lambda ^2)}))^{-1}).\]
Noting $\alpha_{\lambda ,n}^{-(\log c_n^{-1})^{-1/2}}\to \infty$ uniformly over $\lambda$, the overall Bayes risk is hence uniformly lower bounded by
\[ 2\big(1+o(1)\big)\Big(\sum_{\lambda\in\Lambda} \sigma_\lambda ^2\alpha_{\lambda ,n}\log(\alpha_{\lambda ,n}^{-1})\Big).
\]
 The supremum at $n$ is attained for
\[ \alpha_{\lambda ,n}=\exp\Big(\frac{\bar\sigma_n^2}{\sigma_\lambda ^2}\log(e\gamma_n(1-c_n))-1\Big)=e^{-1}(e\gamma_n(1-c_n))^{\bar\sigma_n^2/\sigma_\lambda^2},
\]
where $\bar\sigma_n>0$ is  such that $\sum_\lambda \alpha_{\lambda ,n}=n\gamma_n(1-c_n)$ holds, provided $\alpha_{\lambda ,n}\le c_n$ for all $\lambda$.
The latter condition is fulfilled if $\bar\sigma_n^2\gtrsim\max_\lambda \sigma_\lambda^2$.

Alternatively, we may write $\alpha_{\lambda,n}=n\gamma_n(1-c_n)w_{\lambda,n}$ and the entropy expression becomes
\[ 2\big(1+o(1)\big)n\gamma_n\sup_{w_{\lambda,n}} \Big(\sum_{\lambda\in\Lambda} \sigma_\lambda ^2 w_{\lambda,n}\big(\log(w_{n,\lambda}^{-1})-\log(n\gamma_n)\big)\Big)
\]
where the $w_{\lambda,n}\in[0,(n\gamma_n(1-c_n))^{-1}]$ sum up to one: $\sum_\lambda w_{\lambda,n}=1$. From this representation we immediately infer the lower bound
\[ 2n\gamma_n\log(\gamma_n^{-1})\big(1+o(1)\big)\frac 1n\sum_{\lambda\in\Lambda} \sigma_\lambda ^2 \]
using the uniform weights $w_{\lambda ,n}=1/n$.

Note that for polynomial growth $\sigma_{(i)}^2\sim (n-i)^\beta$, $\beta>0$, and for $r_n=o(n)$, we  have $\sigma_{(r_n)}^2/\sigma_{(1)}^2\to 1$ and the lower bound is indeed
\[ \sup_{f\in{\cal F}_0(\gamma_n)}{\bf E}_f[\norm{\hat{f}_n-f}^2]
\ge 2\big(1+o(1)\big)\norm{\Sigma}n\gamma_n\log(\gamma_n^{-1}).\]
\end{proof}

\section{Minimax upper bound}

\label{sec:upp}

Consider now the setting where the sparsity $\gamma_n$ is known and a correctly tuned threshold estimator is applied in order to identify the unknown positions
of the significant non-zero coefficients $f_\lambda$.

\begin{theorem}
\label{th:upper}
Consider the threshold estimator defined coordinate-wise by
\[ \hat f_\lambda=X_\lambda{\bf 1}_{\{X_\lambda^2>2 \sigma_\lambda^2\log(\alpha_{\lambda,n}^{-1})\}}\text{
with }\alpha_{\lambda,n}:=e^{-\beta_n/\sigma_\lambda^2}
\]
and $\beta_n>0$ chosen such that $\sum_{\lambda \in\Lambda} \alpha_{\lambda,n}=n\gamma_n$. Then, as $n\to\infty$,
\[\sup_{f\in{\cal F}_0(\gamma_n)}{\bf E}_f[\norm{\hat{f}_n-f}^2]\le 2n\gamma_n\beta_n(1+o(1))
\]
holds. This implies that, as $n\to\infty$,
\[
\sup_{f\in{\cal F}_0(\gamma_n)}{\bf E}_f[\norm{\hat{f}_n-f}^2]\le 2n\gamma_n\log(\gamma_n^{-1})\norm{\Sigma}(1+o(1)),
\]
which is minimax optimal for at most polynomial growth in $(\sigma_\lambda^2)$ by the lower bound in Theorem \ref{th:lower}.
\end{theorem}

\begin{remark}
For faster growth than polynomial, we might well have $\beta_n=\log(\gamma_n^{-1})o(\norm{\Sigma})$. So, in general the upper bound matches exactly
the lower bound with respect to the term $2n\gamma_n\log(\gamma_n^{-1})$, while the influence of the heterogeneous noise depends on the specific case.
However, this procedure is non-adaptive since the threshold relies on the knowledge of the sparsity $\gamma_n$.
\end{remark}

\begin{proof}

Introduce the threshold value $\tau_{\lambda,n}=\sqrt{2\log(\alpha_{\lambda,n}^{-1})}$ and note $\max_\lambda\alpha_{\lambda,n}\to 0$.
We can split the error as follows:
\[ \E[(\hat
f_\lambda-f_\lambda)^2]=f_\lambda^2\PP((\xi_\lambda+f_\lambda/\sigma_\lambda)^2\le\tau_{\lambda,n}^2)+\E[\sigma_\lambda^2\xi_\lambda^2{\bf
1}_{\{(\xi_\lambda+f_\lambda/\sigma_\lambda)^2>\tau_{\lambda,n}^2\}}]=:I+II.
\]
For $f_\lambda>\tau_{\lambda,n}\sigma_\lambda$ term I is estimated by
\[ I\le f_\lambda^2\PP(\xi_\lambda\le \tau_{\lambda,n}-f_\lambda/\sigma_\lambda)\le
f_\lambda^2\exp(-(\tau_{\lambda,n}-f_\lambda/\sigma_\lambda)^2/2).
\]
Together with a symmetric argument for $f_\lambda<-\tau_{\lambda,n}\sigma_\lambda$ and a direct bound for $f_\lambda^2\le\tau_{\lambda,n}^2\sigma_\lambda^2$,
we thus obtain a  bound for general $f_\lambda$:
\[ I\le \big(f_\lambda^2\exp(-(\tau_{\lambda,n}-\abs{f_\lambda}/\sigma_\lambda)^2/2)\big)\vee \tau_{\lambda,n}^2\sigma_\lambda^2.
\]
Since for $\tau_{\lambda,n}\to\infty$ we have $\sup_{x\ge 1} x^2e^{-\tau_{\lambda,n}^2(x-1)^2/2}\to 1$, we consider $x=\abs{f_\lambda}/(\tau_{\lambda,n}\sigma_\lambda)$ and infer
\[I\le
\sigma_\lambda^2\tau_{\lambda,n}^2(1+o(1))
\text{ uniformly in $\lambda$}.
\]
Inserting the choice of the thresholds, we conclude
\[ I\le \sigma_\lambda^2\tau_{\lambda,n}^2(1+o(1)){\bf 1}_{\{f_\lambda\not=0\}}
=2\sigma_\lambda^2\log(\alpha_{\lambda,n}^{-1})(1+o(1)){\bf 1}_{\{f_\lambda\not=0\}}.
\]

For term II and $f_\lambda\not=0$ the immediate estimate $II\le
\sigma_\lambda^2$ suffices, while for $f_\lambda=0$ we integrate out explicitly and obtain:
\[ II=\sigma_\lambda^2\E[\xi_\lambda^2{\bf
1}_{\{\xi_\lambda^2>\tau_{\lambda,n}^2\}}]=\sigma_\lambda^22(\tau_{\lambda,n}+1)e^{-\tau_{\lambda,n}^2/2}
=2\sigma_\lambda^2\sqrt{2\log(\alpha_{\lambda,n}^{-1})}\alpha_{\lambda,n}(1+\tau_{\lambda,n}^{-1}).
\]
The overall risk of our estimator is therefore bounded by
\begin{align*}
&\sum_{\lambda\in\Lambda} \E_f[(\hat f_\lambda-f_\lambda)^2]\le
\sum_{\lambda:f_\lambda\not=0}\Big(2\sigma_\lambda^2\log(\alpha_{\lambda,n}^{-1})(1+o(1))
+\sigma_\lambda^2\Big)+\sum_{\lambda:f_\lambda=0}
2\sigma_\lambda^2\sqrt{2\log(\alpha_{\lambda,n}^{-1})}\alpha_{\lambda,n}(1+o(1))\\
&\le (2+o(1))
\Big(\sum_{\lambda:f_\lambda\not=0}\log(\alpha_{\lambda,n}^{-1})\sigma_\lambda^2
+\sqrt 2\max_\lambda\big((\log(\alpha_{\lambda,n}^{-1}))^{-1/2}\alpha_{\lambda,n}\big)\sum_{\lambda:f_\lambda=0}
\sigma_\lambda^2\log(\alpha_{\lambda,n}^{-1})\Big).
\end{align*}
Choosing $\alpha_{\lambda,n}=e^{-\beta_n/\sigma_\lambda^2}$, with $\beta_n>0$ satisfying $\sum_{\lambda \in\Lambda} \alpha_{\lambda,n}=n\gamma_n$,
minimises the last bound (asymptotically) and yields
\[ \sum_{\lambda\in\Lambda} \E_f[(\hat f_\lambda-f_\lambda)^2]\le (2+o(1))n\gamma_n\beta_n
\]
because by $\max_\lambda(\log(\alpha_{\lambda,n}^{-1}))^{-1/2}\alpha_{\lambda,n}\to 0$ the second term is of smaller order.
The last result is a direct consequence.
Indeed, we always have $\beta_n\le\log(\gamma_n^{-1})\norm{\Sigma}$ by bounding $\sigma_\lambda^2\le\norm{\Sigma}$,
which is minimax optimal for at most polynomial growth in $(\sigma_\lambda^2)$ by the lower bound in Theorem \ref{th:lower}.
\end{proof}

\section{A numerical example}
\label{sec:num}

\begin{figure}[t]
\centering
\includegraphics[width=14cm,height=7cm]{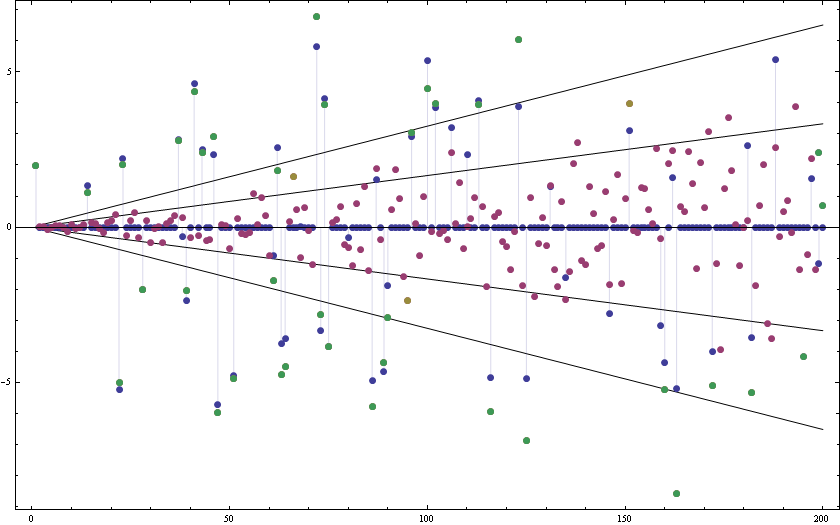}
\caption{Coefficients $(f_\lambda)$ (blue), observations $(X_\lambda)$ (green in full subset, green/yellow in adaptive threshold, magenta not taken) and universal/sparse thresholds (black)
(parameter values: $n=200$, $\gamma_n=0.25$, $\sigma_\lambda=0.01\lambda$ for $\lambda=1\ldots n$).
}
\label{Fig1}
\end{figure}

In Figure \ref{Fig1} a typical realisation of the coefficients $f_\lambda$ is shown in blue with 50 non-zero coefficients chosen uniformly on $[-6,6]$ and increasing noise level $\sigma_\lambda=0.01\lambda$ for $\lambda=1,\ldots,200$. The inner black diagonal lines indicate the sparse threshold (with oracle value of $\gamma_n$) and the outer diagonal lines the universal threshold. The non-blue points depict noisy observations $X_\lambda$. Observations included in the adaptive full subset selection estimator are coloured green, while those included for the adaptive threshold estimator are the union of green and yellow points (in fact, for this sample the adaptive thresholding selects all full subset selected points), the discarded observations are in magenta.

\begin{figure}[t]
\centering
\includegraphics[width=14cm,height=7cm]{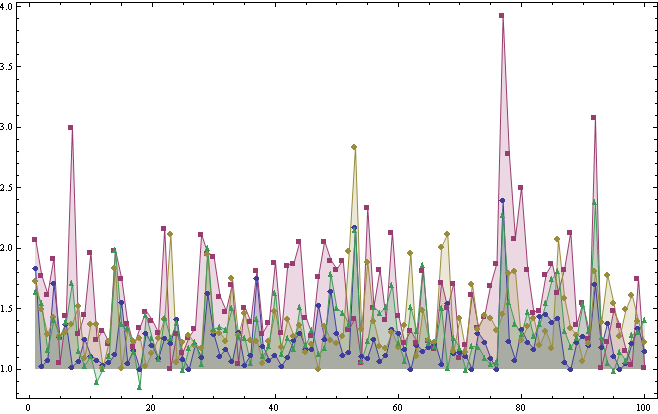}
\caption{First 100 Monte Carlo relative errors: adaptive (blue), universal (magenta) and sparse (yellow) thresholding, full subset selection (green).
}
\label{Fig2}
\end{figure}

We have run 1000 Monte Carlo experiments for the parameters $n=200$, $\sigma_\lambda=0.01\lambda$ in the sparse ($\gamma_n=0.05$) and dense ($\gamma_n=0.25$) case.
In Figure \ref{Fig2} the first 100 relative errors are plotted for the different estimation procedures in the dense case. The errors are taken as a quotient with the sample-wise oracle threshold value applied to the renormalised $X_\lambda/\sigma_\lambda$. Therefore only the full subset selection can sometimes have relative errors less than one. Table \ref{Tab1} lists the relative Monte Carlo errors for the two cases. The last column reports the relative error of the oracle procedure with $h_\lambda={\bf 1}(f_\lambda\not=0)$ that discards all observations $X_\lambda$ with $f_\lambda=0$ (not noticing the model selection complexity).

\begin{table}[t]
{\centering

\begin{tabular}{|l|r|r|r|r|r|r|}\hline
$\gamma_n$ & Adaptive Thr. & Universal Thr. & Sparse Thr. & Full Subset & No Model Selection\\\hline\hline
0.05  & 1.81 & 1.80 & 2.26 & 1.86 & 0.55\\\hline
0.25 & 1.22 & 1.62 & 1.39 & 1.33 & 0.53\\\hline
\end{tabular}\\[3mm]
}
\caption{Relative errors from 1000 Monte Carlo simulations}
\label{Tab1}
\end{table}

The simulation results are quite stable for variations of the setup. Altogether the thresholding works globally well. The (approximate) full subset selection procedure (see below for the greedy algorithm used) is slightly worse and exhibits a higher variability, but is still pretty good. By construction, in the dense case the oracle sparse threshold works better than the universal threshold, while the universal threshold works better in very sparse situations. The reason why the sparse threshold even with a theoretical oracle choice of $\gamma_n$ does not work so well is that the entire theoretical analysis is based upon potentially most difficult signal-to-noise ratios, that is coefficients $f_\lambda$ of the size of the threshold or the noise level. Here, however, the effective sparsity is larger (i.e., effective $\gamma_n$ is smaller) because the uniformly generated non-zero coefficients can be relatively small especially at indices with high noise level, see also Figure \ref{Fig1}.

Let us briefly describe how the adaptive full subset selection procedure has been implemented. The formula \eqref{EqPen} attributes to each selected coefficient $X_\lambda$ the individual penalty $p_\lambda^h=2\sigma_\lambda^2(\log(ne/R_\lambda^h)+\log_+(n\norm{\Sigma})/R_\lambda^h$ with the inverse rank $R_\lambda^h$ of $(h_\lambda\sigma_\lambda^2)_\lambda$. Due to $p_\lambda^h\le 2\sigma_\lambda^2(\log(ne)+\log_+(n\norm{\Sigma}))$ all coefficients with \[X_\lambda/\sigma_\lambda^2\ge 4(\log(ne)+\log_+(n\norm{\Sigma}))\]
are included into $h_1^\ast$ in an initial step. Then, iteratively $h_i^\ast$ is extended to $h_{i+1}^\ast$ by including all coefficients with
\[X_\lambda/\sigma_\lambda^2\ge 4(\log(ne/R_\lambda^{h_i^\ast})+\log_+(n\norm{\Sigma})/R_\lambda^{h_i^\ast}).
\]
The iteration stops when no further coefficients can be included. The estimator $h_I^\ast$ at this stage definitely contains all coefficients also taken by $h^\ast$. In a second iteration we now add in a more greedy way coefficients that will decrease the total penalized empirical risk. Including a new coefficient $X_{\lambda_0}$, adds to the penalized empirical risk the (positive or negative) value
\begin{align*}
& -X_{\lambda_0}^2+4\sigma_{\lambda_0}^2(\log(ne/R_{\lambda_0}^{h_I^\ast})+
\log_+(n\norm{\Sigma}))/R_{\lambda_0}^{h_I^\ast}\\
&\qquad\qquad -4\sum_{\lambda: \sigma_\lambda<\sigma_{\lambda_0}}(h_I^\ast)_\lambda\sigma_\lambda^2
(\log(1+1/R_{\lambda}^{h_I^\ast})
+\log_+(n\norm{\Sigma})/(R_{\lambda}^{h_I^\ast}(R_{\lambda}^{h_I^\ast}+1))).
\end{align*}
Here, $R_{\lambda_0}^{h_I^\ast}$ is to be understood as the rank at $\lambda_0$ when setting $(h_I^\ast)_{\lambda_0}=1$. Consequently, the second iteration extends $h_I^\ast$ each time  by one coefficient $X_{\lambda_0}$ for which the displayed formula gives a negative value until no further reduction of the total penalized empirical risk is obtainable. This second greedy optimisation does not necessarily yield the optimal full subset selection solution, but most often in practice it yields a coefficient selection $h^\ast$ with a significantly smaller penalized empirical risk than the adaptive threshold procedure. The numerical complexity of the algorithm is of order $O(n^2)$  due to the second iteration in contrast to the exponential order $O(2^n)$ when scanning all possible subsets. A more refined analysis of our procedure would be interesting, but might have minor statistical impact in view of the good results for the straight-forward adaptive thresholding scheme.

\bigskip \noindent
{\bf Acknowledgements}

The authors would like to thank Iain Johnstone, Debashis Paul and Thorsten Dickhaus for interesting discussions.
M. Rei{\ss} gratefully acknowledges financial support from the DFG via Research Unit FOR1735 {\it Structural Inference in Statistics}.

%\section{References}

\end{document}